\documentclass[12pt,twoside]{article}

\usepackage{amssymb,amsmath,amsthm}

\date{}

\begin{document}

\centerline{}

\centerline{}

\centerline {\Large{\bf The Criterion of Completely Reducibility}}

\centerline{}

\centerline{\Large{\bf for Continuous Representations of Group Algebras}}

\centerline{}

\centerline{\bf {V.I.Chilin}}

\centerline{}

\centerline{Department of Mathematics,}

\centerline{National University of Uzbekistan,}

\centerline{Vuzgorodok, 100174, Tashkent, Uzbekistan}

\centerline{e-mail: chilin@ucd.uz}

\centerline{}

\centerline{\bf {K.K.Muminov}}

\centerline{}

\centerline{Department of Mathematics,}

\centerline{National University of Uzbekistan,}

\centerline{Vuzgorodok, 100174, Tashkent, Uzbekistan}

\centerline{e-mail: m.muminov@rambler.ru}

\newtheorem{Theorem}{\quad Theorem}[section]

\newtheorem{Definition}[Theorem]{\quad Definition}

\newtheorem{Corollary}[Theorem]{\quad Corollary}

\newtheorem{Lemma}[Theorem]{\quad Lemma}

\newtheorem{Example}[Theorem]{\quad Example}

\begin{abstract} It is shown that every nonsingular continuous representation of the group algebra $L^{1}(G)$ in Banach
spaces is completely reducible if and only if $G$ is a compact group.

\end{abstract}



\section{Introduction} The group algebra $L^{1}(G) = L^{1}(G, \mu)$  , where $G$ is a locally compact group with a left invariant
Haar measure $\mu$, \ is one of the important examples  of *- algebras (see, for example, \cite {M1}, Ch. VI, $\S$ 28).
It is known that any nonsingular continuous *-representation $\pi$ of the *-algebra $L^{1}(G)$  in a Hilbert space $H$
 is generated by the corresponding continuous unitary representation $\rho$ of $G$ in $H$ , and a subspace $L \subset H$,
 is invariant with respect to $\pi$ if and only if $L$ is invariant with respect to $\rho$  ( \cite {M1}, Ch. VI, $\S$ 29).
 Since any unitary representation of $G$ is completely reducible ( \cite {K}, $\S$ 7, 7.3), any nonsingular
 continuous *-representation of the group algebra $L^{1}(G)$ in the Hilbert space
is also completely reducible.

Is the same true for any nonsingular continuous representation of algebra $L^{1}(G)$ in Banach  spaces?  In this case
for the representation of algebra $L^{1}(G)$ one cannot expect a preservation of the involution and, therefore,
the structure of $L^{1}(G)$ as of a Banach algebra without its * - properties is important.

In this paper it is shown that every nonsingular continuous representation of the group algebra $L^{1}(G)$ in Banach
spaces is completely reducible if and only if $G$ is a compact group.  It is also shown that representation $\pi$
of algebra $L^{1}(G)$ in a Banach space is completely reducible if and only if there exists an eigenvector $x \in X$
 for eigenfuncional $F \in X{'}$  such that $F(x)\neq 0$.

We follow notations and terminology of \cite {M1}, \cite {K}, \cite {M2}, \cite {JSH}.

\section{Preliminary Notes}

Let $(G, \tau)$  be a locally compact topological group, $\mu$ be a left invariant Haar measure on
$G, (L_{1}(G), ||\cdot||_{1}$   be a Banach space of all integrable complex functions on $(G, \tau)$. We use
$\int a(g) dg$  instead of $\int_{G} a d \mu, a \in L^{1}(G)$,
and $C(G)$ stands for the linear subspace of all continuous functions from $L^{1}(G)$  with a compact support
(\cite {L}, Ch. I, $\S$ 1). It is known that $C(G)$ is dense in $(L^{1}(G), ||\cdot||_{1})$.

Let $(X, ||\cdot||_{X})$  be any complex Banach space and $X^{'}$ be its conjugate space. Let us denote by $B(X)$
the Banach space of all continuous linear maps from $X$  to $X$ , and by $GL(X)$  the group of all invertible maps
from $B(X)$ . We consider $B(X)$  with respect to a strong operator topology $t_{s}$ .
 Convergence of a net $\left\{T_{\alpha}\right\} \subset B(X)$  to $T \in B(X)$  with respect topology
 $t_{s}$  means that $||T_{\alpha}x - Tx||_{X} \to 0$  for all $x\in X$.

The representation of group $G$ in the Banach space $(X, ||\cdot||_{X})$ is a homomorphism $\rho$  from the group $G$ to
group $GL(X)$. It is said to be strongly continuous if $\rho (g_{\alpha}) \rightarrow \rho(g)$ with respect topology ${t_{s}}$
whenever $g_{\alpha} \rightarrow g$ in $(G, \tau)$.

For any strongly continuous representation  $\rho$ one can define a linear map
$\pi_{\rho} : C(G) \to B(X)$, where by definition,
\begin{equation} \label {fr 1}
\pi_{\rho}(\varphi)(x) = \int_{G}\varphi(g)\rho(g)(x)dg.
\end{equation}

The last integral converges in $X$ because the map $g \mapsto \varphi(g)\rho(g)(x)$ from $(G, \tau)$ to $(X, ||\cdot||_{X})$
is continuous and has a compact support. It is known that $\pi_{\rho}(\varphi)\in B(X)$ for all $\varphi\in C(G)$, and
$\pi_{\rho}$ is a ring homomorphism (\cite {L}, Ch. I, $\S$ 1), i.e.

\begin{equation}  \label {fr2}
\pi_{\rho}(\varphi\ast\psi) = \pi_{\rho}(\varphi)\pi_{\rho}(\psi)
\end{equation}
where $(\varphi\ast\psi)(g) = \int_{G}\varphi(h)\psi(h^{-1}g)dh$.

Assume that $\rho$ is bounded, i.e. there exists such a positive number $\lambda$ that $||\rho(g)||_{B(X)} \leq\lambda$
 for all $g\in G$. In this case the map $\pi_{\rho}$  can be extended to $(L^{1}(G), ||\cdot||_{1}$, moreover,
 $(\ref {fr2})$ is valid, and $||\pi_{\rho}(f)||_{B(X)} \leq\lambda ||f||_{1}$ for all $f\in L^{1}(G)$
 (\cite {L}, Ch. I, $\S$ 1). So $\pi_{\rho}$ is a continuous linear homomorphism from the Banach algebra
 $(L^{1}(G), ||\cdot||_{1}$ to the Banach algebra $(B(X), ||\cdot||_{B(X)})$, i.e. $\pi_{\rho}$
is a continuous representation of algebra $L^{1}(G)$ in the Banach space $X$.

Let $\rho$ be a strongly continuous representation of $(G, \tau)$ in $(X, ||\cdot||_{X})$. A closed linear subspace
$Y$ of $X$ is said to be $\rho$ - invariant ($\pi_{\rho}$ - invariant) if $\rho(g)(Y) \subset Y$
 (resp., $\pi_{\rho}(\varphi)(Y) \subset Y$ ) for all $g\in G$
(resp., $\varphi\in C(G)$). It is known (\cite {L}, Ch. I, $\S$ 1) that a closed subspace $Y$ is $\rho$ - invariant
if and only if it is $\pi_{\rho}$ - invariant.

Nonzero $x\in X \ (F \in X')$ is said to be an eigenvector (resp., an eigenfunctional) for $\rho$ if
$\rho(g)x = \lambda (g)x$ (resp., $F(\rho(g)y) = \lambda (g)F(y)$)  whenever $g \in G, \ y \in X$, where
$\lambda(g) \in C$. It is clear that $\lambda(g)$ is a
continuous homomorphism from $G$ to the unit sphere $\left\{\lambda\in C: |\lambda| = 1\right\}$.

We say that a strongly continuous representation $\rho : (G, \tau) \to (GL(X), t_s)$ is $B$ - representation if for any
 eigenfunctional $F \in X'$ there exists an eigenvector $x\in X$ for $\rho$ such that $F(x) \neq 0$.
 We say that a locally compact group $(G, \tau)$ is a $B$ - group
if its every  bounded strongly continuous representation is a $B$ - representation.

\begin{Theorem} \label{1.1.} The following conditions are equivalent:

(i) $(G, \tau)$ is a   $B$ - group;

(ii) $(G, \tau)$ is a compact group.
\end{Theorem}

\begin{proof} $(i) \Rightarrow (ii)$  Consider a left regular representation $\rho$ of group $G$ in the Banach space
$L^{1}(G): (\rho(g)\varphi)(t) = \varphi(gt), \ \varphi\in L^{1}(G), \ g, t \in G$.

Due to the equality $||\rho(g)\varphi||_{1} = \int_{G}|\varphi(gt)|dt = ||\varphi||_{1}$, the map $\rho(g)$ is an
isometric map and, in particular, $\rho$ is a bounded representation.

Let us show that $\rho$ is a strongly continuous representation.  Let $\varphi\in C(G)$ and $K$ be its compact support.
 Since the function $\varphi$ is uniformly continuous on $K$, for a given $\varepsilon > 0$ there exists a compact
 neighborhood $U$ of the identity element $e \in G$, such that $|\varphi(h) - \varphi(g)|< \varepsilon$,
 whenever $gh^{-1} \in U$. If $g_{1}\in Ug_{0}$ then $(g_{1}t)(g_{0}t)^{-1} \in U$,
and  therefore
\begin{center}
$||\rho(g_{0})(\varphi) - \rho(g_{1})(\varphi)||_{1} =
\int_{(g_{0}^{-1}K)\bigcup(g_{1}^{-1}K)}|\varphi(g_{0}t) - \varphi(g_{1}t)|dt \leq  2\varepsilon\mu(K)$.
\end{center}
So one has $||\rho(g_{\alpha}(\varphi) - \rho(g_{0}(\varphi)||_{1} \to 0$ whenever $g_{\alpha} \rightarrow g_{0}$.

Since $C(G)$ is dense in $(L^{1}(G), ||\cdot||_{1})$, $\rho$ is a strongly continuous representation of the group
$(G, \tau)$ in $(L^{1}(G), ||\cdot||_{1})$. A nonzero continuous linear functional $F(\varphi) = \int \varphi (g) dg$
 is an eigenfuncional for $\rho$.
Consequently, there exists an eigenvector $\varphi_{0} \in L^{1}(G)$ for which $F(\varphi_{0})\neq 0$.
 Due to $\rho(g)\varphi_{0} = \lambda(g)\varphi_{0}$, one has $\lambda(g)F(\varphi_{0}) = F(\rho(g)(\varphi_{0})) =
 F(\varphi_{0})$, i.e. $\lambda(g) = 1$ for all $g\in G$.
It implies that $\varphi_{0}(gt) = \varphi_{0}(t)$, i.e. $\varphi_{0} \equiv const$ and $(G, \tau)$ is a compact group.

$(ii) \Rightarrow (i)$.  Let $(G, \tau)$ be a compact group, $\rho : (G, \tau) \to (B(X), t)$ be  a strongly continuous
representation of $G$ in a Banach space $X$. Let $F$ be an eighenfunctional for $\rho$, i.e.
$F(\rho(g)y) = \lambda (g)F(y)$ for all $g \in G, \ y \in X$.
Take such $x_{1}\in X$ that $F(x_{1}) \neq 0$. The integral $\int\limits_{G}\overline{\lambda(h)}\rho(h)(x_{1})dh =
x_{0} \in X$ converges in $X$ because $\lambda(g)$ is a continuous function on the compact group $(G, \tau)$.
Moreover,
$$\rho(g)x_{0} = \rho(g)\int\limits_{G}\overline{\lambda(h)}\rho(h)(x_{1})dh =
\int\limits_{G}\overline{\lambda(h)}\rho(gh)(x_{1})dh $$
$$= \int\limits_{G}\overline{\lambda(g^{-1}h)}\rho(t)(x_{1})dh =
\overline{\lambda(g^{-1})}x_{0}$$.

Due to $|\lambda(h)| \equiv 1$ one has $$F(x_{0}) = \int\limits_{G}\overline{\lambda(h)}F(\rho(h)(x_{1}))dh =
\int\limits_{G}\overline{\lambda(h)}\lambda(h)F(x_{1}) = F(x_{1})\mu(G) \neq 0,$$
i.e. $x_{0} \neq 0$, and therefore, $x_{0}$ is an eigenvector for $\rho$. So $\rho$ is a $B$ - representation, i.e.
$(G, \tau)$ is a $B$ - group.
\end{proof}

A strongly continuous representation $\rho : (G, \tau) \to (GL(X), t_s)$ is called a $D$ - representation if for any
$\rho$ - invariant closed linear subspace of $X$ there exists  a $\rho$ -invariant closed complement.
A locally compact group $(G, \tau)$ is called a $D$ - group if its every bounded strongly continuous representation is a
$D$ - representation.

It is known (\cite {SH}) that every strongly continuous representation of a compact group in a Banach space is
completely reducible. The next theorem states that every $D$ - group is a compact group.

\begin{Theorem} \label{1.2.}
For any locally compact group $(G, \tau)$ the following conditions are equivalent:

(i) $(G, \tau)$ is a compact group;

(ii) $(G, \tau)$ is a $D$ - group;

(iii) $(G, \tau)$  is a $B$ - group.

\end{Theorem}

\begin{proof}  The implication $(i) \Rightarrow (ii)$  is shown in \cite {SH}  and the implication
$(iii) \Rightarrow (i)$ is true due to Theorem 1.

$(ii) \Rightarrow (iii)$ Let $G$ be a $D$ - group and $\rho : (G, \tau) \to (GL(X), t_s)$ be bounded strongly continuous
representation of $G$ in a Banach space $X$. Let $F$ be an eigenfunctional for $\rho$. It is clear that the closed
subspace $V = kerF = \left\{x\in X: F(x) = 0\right\}$ is $\rho$ - invariant. Since $G$ is a $D$ - group, there exists
a closed $\rho$ - invariant linear subspace $W = \left\{\lambda x_{0}\right\}_{\lambda\in C}$,
such that $X = V\oplus W$, where $0 \neq x_{0}\in X$ and $F(x_{0})\neq 0$. One has $\rho(g)(x_{0}) = \lambda(g)(x_{0})$
due to $\rho(g)(W) \subset W$, where $\lambda(g)\in C$, i.e., $x_{0}$ is an eigenvector for $\rho$,
and $F(x_{0}) \neq 0$. It means that $\rho$ is a $B$ - representation, and therefore, $G$ is a $B$ - group.

\end{proof}

\section{ Completely reducible continuous representations of group algebras}

Let $(G, \tau)$ be a locally compact group, $\rho : (G, \tau) \to (GL(X), t_s)$ be a bounded strongly continuous
representation of $G$ in a Banach space $X$.  It was already noticed that equality (1) defines a representation
$\pi_{\rho}$ of algebra $C(G)$ in $B(X)$, moreover $\pi_{\rho}$ can be considered as a continuous homomorphism  from
the Banach algebra $(L^{1}(G), ||\cdot||_{1})$ to the Banach algebra $(B(X), ||\cdot||_{B(X)})$.
We use the same notation $\pi_{\rho}$ for this extension  and call it the associated representation of algebra
$L^{1}(G)$ to the representation $\rho$ of the group  $(G, \tau)$.

\begin{Lemma}\label{3.1.}
The constructed representation $\pi_{\rho}$ has the following non-singularity property: the set
$\left\{\pi_{\rho}(\varphi)(x): \varphi\in C(G), \ x\in X\right\}$ is dense in $X$.
\end{Lemma}
\begin{proof}  Fix $x\in X$ and $\varepsilon >0$. Use strongly continuity of the representation $\rho$, to get a compact
 neighborhood $U$ of the identity element in $(G, \tau)$ for which $||\rho(g)(x) - x||_{X}< \varepsilon$ for all $g\in U$.
 Consider a nonnegative function $\varphi\in C(G)$ with $supp\varphi\subset U$ for which $\int\limits_{G}
 \varphi(g)dg = 1$. For it one has $\pi_{\rho}(\varphi)(x) - x = \int\limits_{G}\varphi(g)\rho(g)(x)dg - x =
 \int\limits_{U}\varphi(g)(\rho(g)x - x)dg$, and therefore, $||\pi_{\rho}(\varphi)(x) - x||_{X} \leq \int\limits_{U}
 \varphi(g)||\rho(g)x - x||_{X}dg \leq \varepsilon$.
\end{proof}

A representation $\pi$ of algebra $L^{1}(G)$  in $B(X)$ is to be called non-singular whenever the set
$\left\{\pi(\varphi)(x): \varphi\in C, \ x\in X\right\}$  is dense in $X$.
\begin{Theorem} \label{3.2.}
Let $\pi$ be a non-singular continuous representation of the Banach algebra $L^{1}(G)$ in $(B(X), ||\cdot||_{B(X)})$.
 There exists a unique bounded strongly continuous representation $\rho : (G, \tau) \to (GL(X), t_{S})$, for which
 $\pi = \pi_{\rho}$.
\end{Theorem}
\begin{proof} Here we use the method of proof of theorem 1 from (\cite {K}, $\S$ 10, 10.2). Let
$\left\{U_{\alpha}\right\}_{\alpha\in A}$ be a basis of  neighborhoods of the identity element $e \in (G, \tau)$
consisting of compact sets. Consider the following partial order on
$A: \alpha\leq\beta$, if $U_{\beta}\subset U_{\alpha}$.

Let $\left\{\varphi_{\alpha}\right\}_{\alpha\in A}$ be any net of nonnegative functions from $C(G)$ with
$supp\varphi_{\alpha}\subset U_{\alpha}$ and $\int\limits_{G}\varphi_{\alpha}(g)dg = 1$. Consider
$(L_{g}\varphi)(h) = \varphi(g^{-1}h)$. Let us show that $\left\{\pi(L_{g}\varphi_{\alpha})(x)\right\}_{\alpha\in A}$
converges in $(X, ||\cdot||_{X})$ for any $x \in X$. Since $||L_{g}\varphi_{\alpha}||_{1} = 1$ and $||\pi(L_{g}
\varphi_{\alpha})||_{1} \leq 1$ for all $\alpha\in A$, it is enough to show the convergence of
$\left\{\pi(L_{g}\varphi_{\alpha})(x)\right\}_{\alpha\in A}$ for elements $x$
from the dense set
$$M = \left\{\pi(\varphi)(y): \varphi\in C(G), \ y\in X\right\}$$.

Let $\varphi\in C(G), \ y\in X$.  For each $\varepsilon >0$ there exists an
element $\alpha(\varepsilon)\in A$ such that $|\varphi(h) - \varphi(g)|< \varepsilon$ for all $h, g \in G$,
whenever $hg^{-1}\in U_{\alpha(\varepsilon)}$.

Since $supp\varphi_{\alpha}\subset U_{\alpha} \subset U_{\alpha(\varepsilon)}$ for $\alpha\geq\alpha(\varepsilon)$,
we get
$$|((L_{g}\varphi_{\alpha}) * \varphi)(h) - (L_{g}\varphi)(h)| \leq\int\limits_{U_{\alpha(\varepsilon)}}
\varphi_{\alpha}(s)|\varphi(s^{-1}(g^{-1}h)) - \varphi(g^{-1}h)|ds \leq\varepsilon$$
for all $\alpha\geq\alpha(\varepsilon)$.

It is clear that the value of $((L_{g}\varphi_{\alpha}) * \varphi)(h)$ is zero outside of the compact set
$$(g supp\varphi_{\alpha})\cdot supp\varphi\subset(gU_{\alpha(\varepsilon)})\cdot supp\varphi := K(\varepsilon)$$
whenever  $\alpha\geq\alpha(\varepsilon)$,
and therefore, $$||(L_{g}\varphi_{\alpha}) * \varphi - (L_{g}\varphi)||_{1} = \int\limits_{K(\varepsilon)\cup supp L_{g}
\varphi}|((L_{g}\varphi_{\alpha}) * \varphi)(h) - (L_{g}\varphi)(h)|dh \leq$$ $$\varepsilon\left[\mu(K(\varepsilon)) +
\mu (supp L_{g}\varphi) \right].$$
So $$\pi(L_{g}\varphi_{\alpha})(\pi(\varphi)(y)) = \pi((L_{g}\varphi_{\alpha})*\varphi)(y) \xrightarrow \ \pi(L_{g}\varphi)
(y).$$

Let $\rho(g)(x)$ stand for the limit of the net $\left\{\pi(L_{g}\varphi_{\alpha})(x)\right\}_{\alpha\in A}$.
It is clear that $\rho(g)$ is a linear operator on $X$. If $x =\pi(\varphi)(y), \  \varphi\in C(G), \ y\in X$ then
$$||\rho(g)(x)||_{X} = ||\pi((L_{g}\varphi(y)||_{X} \leq \limsup\limits_{\alpha\in A}||\pi((L_{g}\varphi_{\alpha})
||\cdot||x||_{X}\leq ||\pi|| \ ||x||_{X},$$
i.e. $\rho(g)\in B(X)$ and $||\rho(g)||_{B(X)}\leq ||\pi||$ for all $g \in G$. It implies that $\rho$ is a bounded
representation of $G$ in $X$.

Let us show that $\rho$ is a strongly continuous representation. Due to the boundedness of $\rho$, it is enough to show
that convergence $g_{\alpha} \rightarrow g$
implies the convergence $\rho(g_{\alpha})(\pi(\varphi)(y)) \rightarrow \ \rho(g)(\pi(\varphi)(y))$ for all
$\varphi\in C(G), \ y\in X$.

Due to $g_{\alpha} \rightarrow g$, for any compact neighborhood $U$ of identity $e$ there exists $\alpha(U)$
such that $g^{-1}_{\alpha}g\in U$ for all $\alpha\geq\alpha(U)$. So $|\varphi(g^{-1}_{\alpha}h) - \varphi(g^{-1}h)|
< \varepsilon$, if $U$ is chosen in  such a way  that $|\varphi(t) - \varphi(s)| < \varepsilon$ whenever $t^{-1}s \in U$.
 In this case, for $\alpha\geq\alpha(U)$ one has
$$||L_{g_{\alpha}}\varphi - L_{g}\varphi||_{1}\leq\int\limits_{(g_{\alpha} supp\varphi)\cup(g supp\varphi)}|
\varphi(g^{-1}_{\alpha}h) - \varphi(g^{-1}h)|dh \leq 2\varepsilon\mu(supp\varphi),$$
and therefore
$$\rho(g_{\alpha})(\pi(\varphi)(y)) =\pi(L_{g_{\alpha}}\varphi)(y) \xrightarrow \ \pi(L_{g}\varphi)(y) =
\rho(g)(\pi(\varphi)(y)).$$

Now let us show that $\pi_{\rho} = \pi$, where $\pi_{\rho}(\varphi)(x) = \int\limits_{G}\varphi(g)\rho(g)(x)dg$,
\ $ \varphi\in C(G)$.
If $\psi\in C(G), \ y\in X$ we have that $$\pi_{\rho}(\varphi)(\pi(\psi)y) = \int\limits_{G}\varphi(g)\pi(L_{g}
\psi)(y)dg$$
 and
$$\pi(\varphi)(\pi(\psi)(y)) = \pi(\varphi * \psi)(y) = \pi(\int\limits_{G}\varphi(g)\psi(g^{-1}h)dg)(y) =$$
$$\int\limits_{G}\varphi(g)\pi(L_{g}\psi)(y)dg,$$
hence $\pi = \pi_{\rho}$.

Now we show the uniqueness of representation $\rho$, for which $\pi = \pi_{\rho}$. Let $\rho_{1} : (G, \tau) \to (GL(X),
t_{s})$ be another bounded strongly continuous representation such that
$\pi(\varphi)(x) = \pi_{\rho_{1}}(\varphi)(x) = \int\limits_{G}\varphi(g)\rho_{1}(g)(x)dg$ for all $ \varphi\in C(G)$
and $x\in X$. For  $g = e$ one has $$\pi(\varphi_{\alpha})(\pi(\varphi)(y)) = \pi(L_{e}\varphi_{\alpha})(\pi(\varphi)(y))
 \xrightarrow \ \pi(\varphi)(y)$$
for all $\varphi\in C(G), \ y\in X$. Since $\pi$ is a continuous representation and $M$ is dense in $X$, we have
$\pi(\varphi_{\alpha}(x)) \xrightarrow  \ x$ for all $x\in X$. Therefore, $\rho_{1}(g)\pi(\varphi_{\alpha}(x)
 \xrightarrow  \ \rho_{1}(g)(x)$. On the other hand,

$$\rho_{1}(g)\pi(\varphi_{\alpha})(x) = \rho_{1}(g)\int\limits_{G}\varphi_{\alpha}(h)\rho_{1}(h)(x)dh = \int\limits_{G}
\varphi_{\alpha}(h)\rho_{1}(gh)(x)dh =$$
$$= \int\limits_{G}\varphi_{\alpha}(g^{-1}s)\rho_{1}(s)(x)ds = \pi(L_{g}
\varphi_{\alpha})(x)$$.

Thus, $\pi(L_{g}\varphi_{\alpha})(x)\xrightarrow \ \rho_{1}(g)(x)$. Similarly, $\pi(L_{g}\varphi_{\alpha})(x)
\xrightarrow \ \rho(g)(x)$,
and therefore, $\rho = \rho_{1}$.
\end{proof}
The following is a consequence of Lemma \ref{3.1.} and Theorem \ref{3.2.}.

\begin{Corollary}\label{4.1.}
Between bounded strongly continuous representations $$\rho : (G, \tau) \to (GL(X), t_{S})$$ and non-singular continuous
representations
$$\pi : (L^{1}(G), ||\cdot||_{1}) \to (B(X), ||\cdot||_{B(X)})$$ there exists a one to one correspondence
given by formula $$\pi_{\rho}(\varphi)(x) = \int\limits_{G}\varphi(g)\rho(g)(x)dg,$$ where $\varphi\in C(G), \ x\in X$.
\end{Corollary}

\begin{Corollary} \label{4.2.}
Let $\rho$ be a bounded strongly continuous representation of $(G, \tau)$ in $(GL(X), t_{S})$, $Y$ be a closed linear
subspace of $X$. Then $Y$ is $\rho$ - invariant if and only if $Y$ is $\pi_{\rho}$ - invariant.
\end{Corollary}

Now we consider properties of representation $\pi$ of algebra $L^{1}(G)$ dealing with the  existence of eigenvectors
and eigenfunctionals.

For a representation $\pi$ of algebra $L^{1}(G)$ in $B(X)$ the notions of eigenvectors and eigenfunctionals are introduced
 exactly in the same way as in the case of a  representation of group $G$. A nonzero element $x\in X$ (nonzero
 functional $F \in X'$) is said to be an eigenvector (resp. eigenfunctional) for $\pi$, if $\pi(f)(x) = \lambda (f)x$
 (resp. $F(\pi(f)(y)) = \lambda(f)F(y)$) for all $f \in L^{1}(G), \ y \in X$, where $\lambda(f)\in C$.

We call algebra $L^{1}(G)$ a $B$ -algebra, if for any non-singular continuous representations $\pi$ algebra $L^{1}(G)$
in $B(X)$ and for any eigenfunctional $F \in X'$ for $\pi$ there exists an eigenvector $x\in X$ for $\pi$ such that
$F(x) \neq 0$.

\begin{Theorem} \label{3.3.}
The group algebra $L^{1}(G)$ of a locally compact group $(G, \tau)$ is a $B$ - algebra if and only if the group
$(G, \tau)$ is a $B$ - group.
\end{Theorem}

\begin{proof} Let $(G, \tau)$ be a $B$ - group and $\pi$ be any non-singular continuous representation of $L^{1}(G)$ in
$B(X)$. Due to Theorem \ref{3.2.},
there exists a bounded strongly continuous representation $\rho : (G, \tau) \to (GL(X), t_s)$ such that
$\pi = \pi_{\rho}$.

If $F$ is an eigenfunctional for the representation $\pi$, i.e. $F(\pi(f)(x)) = \lambda(f)F(x)$ for all $f \in L^{1}(G),
\ x\in X$ then, in particular, $F(\pi(L_{g}\varphi_{\alpha})(x)) = \lambda(L_{g}\varphi_{\alpha})F(x)$, where the net
$\left\{\varphi_{\alpha}\right\}$
is the same net considered in the proof of Theorem \ref{3.2.}. Since the functional $F$ is continuous and
$\pi(L_{g}\varphi_{\alpha})(x)$
converges to $\rho(g)(x)$ (see the proof of Theorem \ref{3.2.}), there exists   $F(\rho(g)(x)) =
\lim\limits_{\alpha}\lambda(L_{g}\varphi_{\alpha})F(x)$.
It implies the existence of $\lambda(g) := \lim\limits_{\alpha}\lambda(L_{g}\varphi_{\alpha})$, for which the equality
$F(\rho(g)(x)) = \lambda(g)F(x)$ holds.

Since $G$ is a $B$ - group, there exists an eigenvector $x_{0}\in X$ for $\rho$, such that $F(x_{0})\neq 0$.
Due to $\rho(g)(x_{0}) = \gamma(g)x_{0}$, for all $g\in G$ where $\gamma$ is a continuous character on $(G, \tau)$,
one has $f\gamma\in L^{1}(G)$ for all $f\in L^{1}(G)$, and $$\pi(f)(x_{0}) = \int f(g)\rho(g)(x_{0})dg =
\int f(g)\gamma(g)(x_{0})dg =$$ $$(\int f(g)\gamma(g)dg)x_{0} = \nu(f)x_{0},$$
where $\nu(f) = \int f(g)\gamma(g)dg$. It means that $x_{0}$ is an eigenvector for the representation $\pi$.
So $L^{1}(G)$ is a $B$ - algebra.

The proof of the second part can be done in a similar way.
\end{proof}
Theorems \ref{1.1.} and \ref{3.3.} imply the following.

\begin{Corollary} \label{4.3.}
A group algebra $L^{1}(G)$ is a $B$ - algebra if and only if $G$ is a compact group.
\end{Corollary}

We call algebra $L^{1}(G)$ a $D$ - algebra if its every non-singular continuous representations $\pi : L^{1}(G)
\to (B(X), ||\cdot||_{B(X)})$ is completely reducible, i.e. for any  $\pi$ - invariant closed linear subspace of $X$
there exists a $\pi$ - invariant closed complement.

\begin{Theorem} \label{3.4.}
A group algebra $L^{1}(G)$ of a locally compact group $(G, \tau)$ is a $D$ - algebra if and only if $(G, \tau)$ is a
$D$ - group.
\end{Theorem}

It can be proved by the use of Theorem \ref{3.2.} and Corollary \ref{4.2.} in a similar way as the proof of
Theorem \ref{3.3.}.   So due to Theorems \ref{1.2.} , \ref{3.3.}  and \ref{3.4.}, one has:

\begin{Corollary}\label{4.4.}
For a group algebra $L^{1}(G)$ of a locally compact group the following conditions are equivalent:

1) $L^{1}(G)$ is a $B$- algebra;

2) $L^{1}(G)$ is a $D$ - algebra;

3) $G$  is a compact group.
\end{Corollary}

Let us present one more example of a normed algebra for which properties 1) and 2) from Corollary \ref{4.4.} are equivalent.

Consider  a subalgebra $C(G)$ of $L^{1}(G)$. For it, as in the case of algebra $L^{1}(G)$, one can define non-singular
 continuous representation $\pi$ from $C(G)$ to $(B(X), ||\cdot||_{B(X)})$, eigenvectors and eigenfunctionals.

Using the scheme of proofs of Theorems \ref{3.2.}, \ref{3.3.}, and \ref{3.4.} one can have the following
characterization of a compact group $G$.

\begin{Theorem} \label{3.6.}
 For a locally compact group $(G, \tau)$ the following conditions are equivalent:

1) $C(G)$  is a $B$ - algebra;

2) $C(G)$  is a $D$ - algebra;

3) $(G, \tau)$  is a compact group.
\end{Theorem}


\begin{thebibliography}{99}

\bibitem{M1}{M.A. Naimark, \em Normed rings,} Nauka, Moscow, 1968. {Russian}.

\bibitem{K} {A.A. Kirillov, \em Elements of representation theory,} Nauka, Moscow, 1978. (Russian).

\bibitem{M2}{M.A. Naimark, \em Representation theory of groups,} Nauka, Moscow, 1976. (Russian).

\bibitem{JSH} {D.P. Zhelobenko , A.I. Shtern, \em Representation of  Lie groups,} Nauka, Moscow,
1983. (Russian).

\bibitem{L} {S. Lang, \em $SL_{1}(R)$,} Addison-Wesley Pullishing Company, California - London - Amsterdam,
1975.

\bibitem{SH}{K. Shiga, Representations of a compact group on a Banach space,
\em Journal Math. Soc. Japan, } {\bf 7} (1955), 224 - 258.

\end{thebibliography}
\end{document}